\documentclass{article}

\usepackage{amsthm}
\usepackage{amsfonts, amsmath}
\usepackage{graphicx}
\usepackage{epstopdf}
\usepackage{algorithmic}

\ifpdf
  \DeclareGraphicsExtensions{.eps,.pdf,.png,.jpg}
\else
  \DeclareGraphicsExtensions{.eps}
\fi

\title{A New Optimality Property of Strang's Splitting}

\author{Fernando Casas\footnote{Departament de Matem\`atiques and IMAC, Universitat Jaume I, E-12071, Castell\'on de la Plana, Spain
  ({casas@uji.es}, {shaw@uji.es}).}
\and Jes{\'u}s Mar{\'\i}a Sanz-Serna\footnote{Departamento de Matem\'aticas, Universidad Carlos III de Madrid, E-28911, Legan\'es, Spain
  ({jmsanzserna@gmail.com}).}
\and Luke Shaw\footnotemark[1]}

\usepackage{amsopn}

\newcommand{\R}{\mathbb{R}}
\usepackage{hyperref}
\graphicspath{{.}}
\usepackage{subcaption}
\usepackage[capitalize]{cleveref}
\usepackage{bbm}
\usepackage{resizegather}

\newtheorem{theorem}{Theorem}[section]
\newtheorem{lemma}[theorem]{Lemma}

\newtheorem{proposition}[theorem]{Proposition}

\usepackage[margin=1in]{geometry}
\usepackage{pgfplots}
\pgfplotsset{width=10cm,compat=1.9}

\usepgfplotslibrary{external}
\begin{document}
\maketitle
\begin{abstract} For systems of the form $\dot q = M^{-1} p$, $\dot p = -Aq+f(q)$, common in many
applications, we analyze splitting integrators based on the (linear/nonlinear) split systems $\dot q =
M^{-1} p$, $\dot p = -Aq$ and $\dot q = 0$, $\dot p = f(q)$. We show that the well-known Strang splitting
is optimally stable in the sense that, when applied to a relevant model problem, it has a larger stability
region than alternative integrators. This generalizes a well-known property of the common
St\"{o}rmer/Verlet/leapfrog algorithm, which of course arises from Strang splitting based on the
(kinetic/potential) split systems $\dot q = M^{-1} p$, $\dot p = 0$ and $\dot q = 0$, $\dot p = -Aq+f(q)$.
\end{abstract}

\noindent {This paper is dedicated to Gilbert Strang.}

\section{Introduction}
We are concerned with numerical integrators for second-order systems in $\R^d$
\begin{equation}\label{eq:SplitEq}
M\ddot{q}=-Aq+f(q),
\end{equation}
where $M$ and $A$ are constant $d\times d$ matrices ($M$ invertible), or equivalently for first-order systems in $\R^{2d}$
$$
\dot q = M^{-1} p, \qquad \dot p = -Aq+f(q).
$$
 Our aim is to prove that the Strang splitting integrator \cite{Strang1963} based on the (linear/nonlinear) split systems
\begin{equation}\label{eq:RotSystem}
\dot q = M^{-1} p,\qquad \dot p = -Aq
\end{equation}
and
\begin{equation}\label{eq:KickSystemf}
\dot q = 0,\qquad \dot p = f(q)
\end{equation}
possesses an optimal stability property.

The format \eqref{eq:SplitEq} is a particular instance of the system
\begin{equation}\label{eq:SplitEqg}
M \ddot q = g(q)
\end{equation}
that appears very frequently in many applications. The best-known integrator for \eqref{eq:SplitEqg} is perhaps the St\"{o}rmer/leapfrog/Verlet algorithm \cite{HLWBook}. In its Verlet formulation, the integrator is constructed by  applying Strang's splitting to the first-order system
$$
\dot q = M^{-1} p, \qquad \dot p = g(q),
$$
with the (kinetic/potential) split systems
\begin{equation}\label{eq:DriftSystem}
\dot q = M^{-1} p,\qquad \dot p = 0,
\end{equation}
and
\begin{equation}\label{eq:KickSystem}
\dot q = 0,\qquad \dot p = g(q).
\end{equation}
More precisely, let us denote by $\varphi_t^{[D]}$ the solution flow of \eqref{eq:DriftSystem},  $t\in \R$,
$$
\varphi^{[D]}_t(q,p) = (q+tM^{-1}p,p),
$$
and
by $\varphi_t^{[K]}$ the solution flow of \eqref{eq:KickSystem},
$$
\varphi_t^{[K]}(q,p) = (q,p+tg(q)),
$$
then a timestep of length $h>0$ of the position Verlet algorithm is given by the map
$$
\psi_h^{[pos]} = \varphi^{[D]}_{h/2}\circ\varphi^{[K]}_h\circ\varphi^{[D]}_{h/2}
$$
and a step
of
the velocity Verlet algorithm is defined by the map
$$
\psi_h^{[vel]} = \varphi^{[K]}_{h/2}\circ\varphi^{[D]}_h\circ\varphi^{[K]}_{h/2},
$$
where the roles of $\varphi^{[D]}$ and $\varphi^{[K]}$ have been swapped. The labels D and K we have used
correspond to the words \emph{drift} and \emph{kick}, commonly used in molecular dynamics to refer to
$\varphi^{[D]}$ and $\varphi^{[K]}$ respectively \cite{Garcia1998}.

In spite of its simplicity, the Verlet integrator is the method of choice in many applications \cite{LeimkuhlerBook}. One of the advantages of the (position or velocity) Verlet integrator is that it possesses, among a wide class of \emph{explicit} integrators, an \emph{optimal} stability interval \cite{Jeltsch1981,Chawla1981,SS1986,BouRabee2018}. In fact,  Verlet strictly maximizes the scaled length of the stability interval, i.e.\ the quotient $\Lambda/m$, where $\Lambda$ is the length of the stability interval and $m$ the number of evaluations of $g$ per step. In other words, for any explicit competitor integrator using $m$ evaluations per step, there are values of $h$ such that Verlet integrations with steplength $h$ are stable while the (equally costly) integrations of the competitor with steplength $mh$ are unstable.
 In short, the Verlet algorithm
may be operated with longer (scaled) timesteps than any of its explicit
competitors; this makes it appealing in applications, including molecular dynamics,
where integrations are performed with values of $h$ close to the stability limit because high accuracy
is either not required or impossible to achieve due to the complexity of the problem (for instance in cases
 where $g$ is very expensive to evaluate).

When, in \eqref{eq:SplitEqg}, $g$ takes the particular form $g(q) = -Aq+f(q)$ as in \eqref{eq:SplitEq}, instead of splitting the given system as \eqref{eq:DriftSystem}--\eqref{eq:KickSystem}, it may be advantageous to split as \eqref{eq:RotSystem}--\eqref{eq:KickSystemf} and consider the Strang integrators RKR and KRK
\begin{equation}\label{eq:RKR}
\psi^{[RKR]}_h = \varphi_{h/2}^{[R]}\circ\varphi_{h}^{[K]}\circ\varphi_{h/2}^{[R]}
\end{equation}
and
\begin{equation}\label{eq:KRK}
\psi^{[KRK]}_h = \varphi_{h/2}^{[K]}\circ\varphi_{h}^{[R]}\circ\varphi_{h/2}^{[K]},
\end{equation}
where $\varphi_t^{[R]}$ and  $\varphi_t^{[K]}$ denote respectively the solution flows of the systems
\eqref{eq:RotSystem} and \eqref{eq:KickSystemf}. (Of course, kicks are now based on $f$ rather than on $g$.)
We use the identifier R from \emph{rotation} because in typical applications the matrices $M$ and $A$ are
symmetric and positive definite and then the solution map
$$
\begin{bmatrix}q\\p\end{bmatrix}\mapsto\exp\left(t\begin{bmatrix}0&M^{-1}\\-A&0\end{bmatrix}\right)
\begin{bmatrix}q\\p\end{bmatrix}
$$
of \eqref{eq:RotSystem} describes, after a suitable linear change of variables, $d$ rotations in the (two-dimensional)  planes
$(q^i,p^i)$, $i=1,\dots,d$, where $q^i$ and $p^i$ are the scalar components of $q$ and $p$. The splitting
\eqref{eq:RotSystem}--\eqref{eq:KickSystemf}  is particularly appealing when, in $g(q) = -Aq+f(q)$,
$f(q)$ is a small perturbation of $-Aq$: RKR, KRK and other splitting algorithms {using sequences of rotations and kicks} are  exact if
the perturbation vanishes. The main contribution of this paper is to show that, as  is the case for the
velocity and position Verlet integrators, the RKR and KRK integrators \eqref{eq:RKR}--\eqref{eq:KRK} possess
an optimal stability property. Roughly speaking, we show that for a model test problem, for each given steplength, RKR and KRK remain stable for larger perturbations $f$ than any other {rotation/kick splitting} integrator (see
Section~\ref{sec:main} for a precise statement).

{\bf Motivation.} Our interest in problems of the form \eqref{eq:SplitEq} originated  when studying
integrators for the Hamiltonian Monte Carlo (HMC) method, a sampling technique widely used in statistics and
statistical physics \cite{Neal2011,SanzSerna2014}. The bulk of the computational effort in HMC is in
integrating systems of the form \eqref{eq:SplitEqg} where $g(q)$ is the negative gradient of the logarithm of
the target probability density function and $M$ is a positive-definite symmetric matrix \emph{chosen by the
user}. Therefore devising suitable efficient integrators is of key importance to HMC
\cite{Blanes2014,BouRabee2018}. In many situations of interest \cite{Shahbaba2014}, the target density is a
perturbation of a Gaussian density and then $g(q) = -Aq+f(q)$ with $A$ the symmetric positive-definite
precision matrix of the Gaussian distribution and $f(q)$ a perturbation. As shown in \cite{Casas2022}, it is
then very advantageous to choose $M= A$ and then \eqref{eq:SplitEq} becomes
\begin{equation}\label{eq:HMCSys}
\ddot q = -q + \bar{f}(q),\qquad \bar{f}(q) = A^{-1}f(q).
\end{equation}
{It is also shown in \cite{Casas2022} that to integrate \eqref{eq:SplitEq} or \eqref{eq:HMCSys} the Strang splitting is far more efficient when applied to \eqref{eq:RotSystem}--\eqref{eq:KickSystemf} than when applied to the kinetic/potential \eqref{eq:DriftSystem}--\eqref{eq:KickSystem}. This suggests the investigation of rotation/kick splitting algorithms for \eqref{eq:SplitEq} or \eqref{eq:HMCSys}. Furthermore, f}or reasons detailed in \cite{Beskos2011,Beskos2013}, as a rule, integrations of \eqref{eq:HMCSys} within HMC simulations are best carried out with values of $h$ close to the stability limit of the integrator. Therefore it is of {clear} interest to identify the {rotation/kick splitting} integrators with optimal stability interval. In fact the motivation for the present research originated when our multiple attempts to construct integrators that improved on KRK or RKR failed \cite{Casas2022}.

{Exponential integrators \cite{Hochbruck2010} are a well-known class of algorithms that, as splitting methods, exploit the structure of \eqref{eq:SplitEq} or \eqref{eq:HMCSys}. However they are not relevant to HMC applications where symplecticness and time-reversibility are essential \cite{BouRabee2018}.}

{\bf Contents.} The article has five sections. Section~\ref{sec:preliminaries} contains preliminary material. The main optimality result is presented and proved in Section~\ref{sec:main}. Section~\ref{sec:discussion} provides complementary results to compare the size of the stability regions of the Strang splitting algorithms and some possible competitors. The final section contains  a technical proposition.

\section{Preliminaries}
\label{sec:preliminaries}

In this section we present a number of facts that are required to formulate and prove the main result presented in the next section.

\subsection{Splitting integrators}
The importance of splitting integrators in different applications has increased substantially in recent decades \cite{BlanesBook}, often in connection with preservation of geometric properties, such as symplecticness \cite{SSBook}.
Of course, the RKR and KRK methods \eqref{eq:RKR} and \eqref{eq:KRK} are not the only splitting algorithms to integrate \eqref{eq:SplitEq} with the help of the split systems \eqref{eq:RotSystem} and \eqref{eq:KickSystemf}. One may consider $m$-stage integrators by interleaving rotations and kicks, beginning with either R or K as follows
\begin{equation}\label{eq:OscillationCompositionIntegrator}
\psi_{h}=\varphi_{r_{m+1}h}^{[R]}\circ\varphi_{k_mh}^{[K]}\circ\varphi_{r_mh}^{[R]}\circ\ldots\circ\varphi_{k_1h}^{[K]}\circ\varphi_{r_1h}^{[R]},
\quad
\psi_{h}=\varphi_{k_{m+1}h}^{[K]}\circ\varphi_{r_mh}^{[R]}\circ\varphi_{k_mh}^{[K]}\circ\ldots\circ\varphi_{r_1h}^{[R]}\circ\varphi_{k_1h}^{[K]}.
\end{equation}
 We always assume the consistency requirements $\sum_i r_i = 1$ and $\sum_i k_i =1$. Some of the coefficients $r_i$ or $k_i$ are allowed to vanish as this simplifies the presentation. Note that the first format in \eqref{eq:OscillationCompositionIntegrator} uses (at most) $m$ kicks and therefore (at most) $\leq m$ evaluations of $f$ per step; the second format uses $\leq m+1$ kicks, but, since, if $k_{m+1}\neq 0$ and $k_1\neq 0$, the value of $f$ at the last kick of the current timestep may be used to perform the first kick of the next timestep, also requires essentially $\leq m$ evaluations of $f$ per timestep.

If $M$ and $A$ are symmetric and positive definite and $f(q) = -\nabla V(q)$ for a suitable scalar function $V$, then \eqref{eq:SplitEq} is  equivalent to the Hamiltonian system with Hamiltonian function $(1/2) p^TM^{-1}p+(1/2) q^TAq+V(q)$. In this case the split systems  \eqref{eq:RotSystem} and \eqref{eq:KickSystemf} are also Hamiltonian and therefore $\varphi_t^{[R]}$ and $\varphi_t^{[K]}$ are, for each $t\in \R$, symplectic maps, as flows of Hamiltonian systems. It follows that the splitting integrators in \eqref{eq:OscillationCompositionIntegrator} will be symplectic, as  is required in HMC applications
\cite{BouRabee2018}.

It is often the case that the coefficients $r_i$, $k_i$ in \eqref{eq:OscillationCompositionIntegrator} are chosen \emph{palindromically}, i.e.\ for compositions starting with $R$, $r_{m+2-i} = r_i$, $i=1,\dots, m+1$, and $k_{m+1-j} = k_j$, $j= 1,\dots, m$, and similarly for compositions starting with $K$. RKR and KRK are both palindromic. Palindromic splitting integrators have at least second order of accuracy and, in addition, are time-reversible, as  required in HMC applications \cite{BouRabee2018}.
\subsection{Conjugate integrators}
 Given two integrators $\psi_h$ and $\bar{\psi}_h$ of the form \eqref{eq:OscillationCompositionIntegrator}, we say that they are conjugate if there is an invertible map $\chi_h$ such that
 $$
 \bar{\psi}_h = \chi_h\circ\psi_h\circ\chi_h^{-1}.
 $$
 This notion goes back to Butcher's algebraic theory of Runge-Kutta methods \cite{Butcher1969, Butcher1972, Butcher1996}. The $n$-fold composition map $\bar{\psi}_h^n$ used to advance $n$ steps with method $\bar{\psi}_h$ may be written as
 $$
 \bar{\psi}_h^n = (\chi_h\circ\psi_h\circ\chi_h^{-1}) \circ (\chi_h\circ\psi_h\circ\chi_h^{-1}) \circ \cdots \circ(\chi_h\circ\psi_h\circ\chi_h^{-1}) = \chi_h\circ\psi_h^n\circ\chi_h^{-1},
 $$
and therefore to advance $n$ steps with method $\bar{\psi}_h$ one may (i) apply once the map $\chi_h^{-1}$ (preprocessing), (ii) advance $n$ steps with the integrator $\psi_h$, (iii) apply once the map $\chi_h$ (postprocessing). Butcher was interested in the case where $\bar{\psi}_h$ has order of consistency higher than $\psi_h$, since then pre/postprocessing make it possible to perform high-order integrations with $\bar{\psi}_h$ by implementing the low-order integrator $\psi_h$.

 An example of conjugate methods is afforded by the integrators RKR and KRK with the postprocessor
 $\chi_h = \varphi_{h/2}^{[R]} \circ \varphi_{h/2}^{[K]}$:
 \begin{eqnarray*}
 \psi^{[RKR]}_h
  &=&
   \varphi_{h/2}^{[R]}\circ\varphi_{h}^{[K]}\circ\varphi_{h/2}^{[R]}\\
   & = &
 \Big( \varphi_{h/2}^{[R]} \circ \varphi_{h/2}^{[K]}\Big)\circ
 \Big( \varphi_{h/2}^{[K]} \circ \varphi_{h}^{[R]}\circ\varphi_{h/2}^{[K]}\Big)\circ
 \Big( \varphi_{h/2}^{[R]} \circ \varphi_{h/2}^{[K]}\Big)^{-1}\\
 &=&
 \chi_h \circ \psi^{[KRK]}_h \circ \chi_h^{-1}.
 \end{eqnarray*}

 One may prove by means of similar manipulations that all (consistent) one-stage integrators, including the non palindromic, first-order  Lie-Trotter
 integrators $ \varphi_h^{[R]}\circ \varphi_h^{[K]}$ and $ \varphi_h^{[K]}\circ \varphi_h^{[R]}$
  may be conjugated to either RKR or KRK, which are palindromic and second-order. Clearly,
  $ \varphi_h^{[R]}\circ \varphi_h^{[K]}$ is obtained by setting $r_2 = 1$, $k_1=1$, $r_1=0$ in the first equality in \eqref{eq:OscillationCompositionIntegrator}; $ \varphi_h^{[K]}\circ \varphi_h^{[R]}$ results from the choice $r_2=0$, $k_1=1$, $r_1=1$ in the same equality. Both integrators may also be obtained by using the format in the second equality in \eqref{eq:OscillationCompositionIntegrator}.

  It is proved in \cite{Blanes2008} that every integrator may be conjugated to a palindromic integrator.

  For each problem \eqref{eq:SplitEq} the numerical trajectory $\psi_h^n(q,p)$, $n= 0, 1, 2,\dots$, generated by $\psi_h$ with initial condition $(q,p)$ is mapped by $\chi_h$ into the trajectory $\bar{\psi}_h^n(q^*,p^*)$, $n= 0, 1, 2,\dots$, with initial condition $(q^*,p^*) =\chi_h(q,p)$. For this reason the long-time properties of the numerical solutions generated by $\psi_h$ and $\bar{\psi}_h$ may be expected to be similar (for instance bounded/unbounded trajectories of $\psi_h$ correspond to bounded/unbounded trajectories of $\bar{\psi}_h$).

\subsection{The model problem}
Roughly speaking, a numerical integration with a given integrator and steplength $h$ is said to be unstable if the numerical solution shows unphysical growth as the number of computed timesteps increases. In order to make this notion mathematically precise, it is standard to restrict the attention to integrations performed on an easy-to-analyse  \emph{model} problem chosen in such a way that conclusions based on the model are relevant when dealing with more general problems.

For \eqref{eq:SplitEqg}, it is standard to use the model scalar problem  $\ddot q = -\omega^2 q$, i.e.\  the familiar harmonic oscillator. The relevance of this choice of model problem may be justified as follows. Let us assume, for simplicity, that $M$, as is the case in most applications, is symmetric and positive-definite (this hypothesis may be relaxed). Writing $M= LL^T$ and introducing new variables $\bar{q} = L^Tq$, \eqref{eq:SplitEqg} becomes $\ddot{\bar{q}} = L^{-1}g(L^{-T}\bar q)$. Furthermore, if $g$ is linear, $g(q) = -Aq$, then $\ddot{\bar{q}} = -L^{-1}AL^{-T}\bar q$. The important case, with oscillatory solutions, is that where $L^{-1}AL^{-T}$  is diagonalizable with positive eigenvalues (which happens if in particular $A$ is symmetric and positive definite). Then a new change of variables  reduces the system to a set of $d$ uncoupled scalar harmonic oscillators $\ddot q = -\omega^2 q$ (the eigenvalues of $L^{-1}AL^{-T}$ provide the values of $\omega^2$). For this construction to be useful it is required that the transformations that diagonalize the system being integrated also diagonalize the integrator, something that invariably happens for all integrators of practical interest.

In order to identify a suitable model problem for integrators for \eqref{eq:SplitEq} we proceed similarly. We consider the case where $f$ is linear $f(q) = -Bq$; the change of variables $\bar{q} = L^Tq$ brings the system to the form $\ddot{\bar{q}} = -L^{-1}(A+B)L^{-T}\bar q$. Under the hypothesis that there is a linear transformation that brings both $L^{-1}AL^{-T}$ and
$L^{-1}BL^{-T}$ to diagonal form,  after  a new change of variables the system is transformed into $d$ uncoupled scalar equations of the form
\begin{equation}\label{eq:ModelPrevious}
\ddot q=-(\lambda+\mu) q,\end{equation} where $\lambda$ and $\mu$ are eigenvalues of $L^{-1}AL^{-T}$  and $L^{-1}BL^{-T}$ associated with the same eigenvector. We are interested in problems with  $\lambda>0$ and $\lambda+\mu>0$ (something which happens in the important case where $A$ and $A+B$ are symmetric and positive definite), so that the equations \eqref{eq:ModelPrevious} corresponds to harmonic oscillators. The analysis of \eqref{eq:ModelPrevious} is simplified if we introduce a new time variable $t/\sqrt{\lambda}$, so as to have, after denoting $\varepsilon = \mu/\lambda$,
\begin{equation}\label{eq:ModelProblem}
\ddot{q}=-q-\varepsilon q,\qquad \varepsilon>-1.
\end{equation}
This model problem, that we refer to hereafter as ``the model problem'', has appeared e.g.\ in \cite{BouRabee2017}.

In the particular situation of the system \eqref{eq:HMCSys} arising in the HMC method, the derivation just outlined of the model
\eqref{eq:ModelProblem} may be greatly simplified. In fact, if $f$ is linear, $f(u)=-Bu$
so that $\bar{f}(u) = A^{-1}Bu$, and $A^{-1}B$ diagonalizes with eigenvalues $\varepsilon>-1$, then a single change of variables reduces \eqref{eq:HMCSys} to $d$ uncoupled harmonic oscillators of the form
\eqref{eq:ModelProblem}. In the case where $f(u)=-Bu$ is a small perturbation of $Au$, the eigenvalues $\varepsilon$ will actually have small magnitude.

\subsection{Integrating the model problem. Stability}
For the model problem \eqref{eq:ModelProblem},
$$
\varphi_t^{[R]}(q,p)=\begin{bmatrix}
\cos(t) &\sin(t)\\-\sin(t) & \cos(t)
\end{bmatrix}\begin{bmatrix}
q\\p
\end{bmatrix},\quad\quad
\varphi_t^{[K]}(q,p)=\begin{bmatrix}
1 &0\\-t\varepsilon & 1
\end{bmatrix}\begin{bmatrix}
q\\p
\end{bmatrix},
$$
where we note that both transformations have unit determinant as each corresponds to the flow of a Hamiltonian
system. By multiplying the matrices that represent the flows being composed in
\eqref{eq:OscillationCompositionIntegrator}, we obtain the matrices representing one step of the splitting integrator
$\psi_h$. In particular for the Strang splittings \eqref{eq:RKR} and \eqref{eq:KRK}, we find that the matrices
that perform a timestep of length $h$ are
\begin{equation}\label{eq:MatrixRKR}
    \begin{bmatrix}
    \cos(h)-\frac{h\varepsilon}{2}\sin(h) & \sin(h)-\varepsilon h\sin^2\left(\frac{h}{2}\right)\\
    -\sin(h)-\varepsilon h\cos^2\left(\frac{h}{2}\right) & \cos(h)-\frac{h\varepsilon}{2}\sin(h)
    \end{bmatrix}\qquad {\rm for}\qquad \psi^{[RKR]}_{\varepsilon,h}
\end{equation}
and
\begin{equation}\label{eq:MatrixKRK}
    \begin{bmatrix}
    \cos(h)-\frac{h\varepsilon}{2}\sin(h) & \sin(h)\\
    -\varepsilon h\cos(h)-(1-\left(\frac{h\varepsilon}{2}\right)^2)\sin(h)
    & \cos(h)-\frac{h\varepsilon}{2}\sin(h)
    \end{bmatrix}\qquad {\rm for}\qquad \psi^{[KRK]}_{\varepsilon,h}.
\end{equation}

For the  integrators in \eqref{eq:OscillationCompositionIntegrator} the (real) matrix  takes the form
$$M_{\varepsilon,h}=
\begin{bmatrix}
    A_{\varepsilon,h} & B_{\varepsilon,h}\\
    C_{\varepsilon,h}& D_{\varepsilon,h}
    \end{bmatrix}.
$$
 The dependence of the coefficients $A-D$ on $\varepsilon$
is \emph{polynomial} and with $m$ stages  $A$ and $D$ are polynomials of degree $\leq m$ in $\varepsilon$ (this is easily proved by induction).
The dependence on $h$, on the other hand, involves both powers of $h$ and trigonometric functions, as illustrated by \eqref{eq:MatrixRKR} and \eqref{eq:MatrixKRK}. For palindromic compositions
$A_{\varepsilon,h}=D_{\varepsilon, h}$ (see e.g.\ \cite{Blanes2014,Campos2017}).

The matrix $M_{\varepsilon,h}$ has unit determinant, as it results from multiplying rotations and kicks of unit
determinant. Then its (possibly complex) eigenvalues are inverse to one another, $\lambda_{\varepsilon,h}$ and $1/\lambda_{\varepsilon,h}$, and it is well known that one of the three following situations obtains:
\begin{enumerate}
\item The modulus of the trace $A_{\varepsilon,h}+D_{\varepsilon,h}=\lambda_{\varepsilon,h}+1/\lambda_{\varepsilon,h}$ of $M_{\varepsilon,h}$ is $<2$. This
    corresponds to two different complex eigenvalues of unit modulus. As $n$ increases the powers $M_{\varepsilon,h}^n$
    remain bounded and the integration is \emph{stable}.
\item The modulus of the trace is $=2$. Then there is a double real eigenvalue $\lambda=1/\lambda\in\{-1,1\}$. If, in
    addition $M_{\varepsilon,h}$ diagonalizes, then $M_{\varepsilon,h}$ is either  $I$ (the identity matrix) or  $-I$,
    with bounded powers, and the integration is \emph{stable}. When $M_{\varepsilon,h}$ does not diagonalize
    its powers grow linearly and the integration is \emph{linearly unstable}.
\item The modulus of the trace is $>2$. Then there is one real eigenvalue of modulus $>1$, leading to
    \emph{exponential instability}.
\end{enumerate}

Cases 1 and 3 above are robust against perturbations, in the sense that if, for a given integrator, the pair
$(\varepsilon, h)$ is in case 1 (respectively, case 3), all sufficiently close pairs are also in case 1
(respectively, case 3). Perturbations of case 2, on the contrary, will generically lead to either case 1 or
case 3. The stability region of an integrator is  the set in the $(\varepsilon,h)$ plane where it is stable.

The semitrace $$P(\varepsilon,h) = (1/2)(A_{\varepsilon,h}+D_{\varepsilon,h})=(1/2)(\lambda_{\varepsilon,h}+1/\lambda_{\varepsilon,h})$$ of $M_{\varepsilon,h}$ will be called, using a not very precise terminology, the \emph{stability polynomial} of the integrator; recall that it is a polynomial in $\varepsilon$ of degree $\leq m$ but its dependence on $h$ includes trigonometric functions. Exponentially unstable integrations correspond then to $|P(\varepsilon,h)|> 1$.

If the integrators $\psi_h$ and $\bar{\psi}_h$ are conjugate to each other, then the corresponding matrices satisfy the similarity condition
$$
\bar{M}_{\varepsilon,h}= S_{\varepsilon,h} M_{\varepsilon,h}S_{\varepsilon,h} ^{-1}
$$
where the matrix $S_{\varepsilon,h}$ corresponds to the postprocessor. As a consequence $\bar{M}_{\varepsilon,h}$ and ${M}_{\varepsilon,h}$ share the same pair of eigenvalues $\lambda_{\varepsilon,h}$, $1/\lambda_{\varepsilon,h}$ and therefore \emph{conjugate integrators share a common stability polynomial.}
  This property is illustrated by the RKR, KRK pair in \eqref{eq:MatrixRKR}--\eqref{eq:MatrixKRK}. The property was perhaps to be expected, because it was  pointed out above that for any two conjugate integrators   the numerical trajectories of  one of them  are mapped by the processor into numerical trajectories of the other.

\subsection{A property of the stability polynomial}

The following result will be essential to prove our main result.

\begin{proposition}\label{prop:consistency}
For each (consistent) integrator \eqref{eq:OscillationCompositionIntegrator} the stability polynomial satisfies:
\begin{equation}\label{eq:OscillationConsistencyRestriction}
P(\varepsilon,h)=\frac{1}{2}(A_{\varepsilon,h}+D_{\varepsilon,h})=\cos(h)-\frac{\varepsilon h}{2}\sin(h)+\mathcal{O}(\varepsilon^2),\qquad \varepsilon\rightarrow 0.
\end{equation}
\end{proposition}
\begin{proof}It is sufficient to consider the R-first format in the first equality in \eqref{eq:OscillationCompositionIntegrator}; a K-first integrator  may be rewritten in the R-first format by adding dummy rotations of duration $0h$ at the beginning and end of the step. We introduce the matrices
$$
R=\begin{bmatrix}
    0&1\\
    -1&0
    \end{bmatrix},
\qquad
    K=\begin{bmatrix}
    0&0\\
    -1&0
    \end{bmatrix},
$$
whose exponentials represent the rotation and the kick
$$
\exp(tR) = \begin{bmatrix}
    \cos(t)&\sin(t)\\
    -\sin(t)&\cos(t)
    \end{bmatrix},\qquad \exp(tK) = I+tK = \begin{bmatrix}
    1&0\\
    -t&1
    \end{bmatrix}.
$$
Then the matrix associated with the integrator is
\begin{equation}\label{eq:matrixasproduct}
M_{\varepsilon,h} = \exp(hr_{m+1}R) (I+\varepsilon h k_mK)\exp(hr_{m}R) (I+\varepsilon hk_{m-1}K)\cdots(I+\varepsilon hk_1K)\exp(hr_1R),
\end{equation}
which leads to
$$
M_{\varepsilon,h} = \exp(h\theta_{m+1}R) +\varepsilon h \sum_{i=1}^m k_i\exp(h(1-\theta_i)R) K \exp(h\theta_iR)+\mathcal{O}(\varepsilon^2),
$$
where $\theta_i = \sum_{j=1}^i r_j$.
By consistency $\theta_{m+1}= 1$ and therefore the semitrace of $\exp(h\theta_{m+1}R)$ is $\cos(h)$; this gives the term independent of $\varepsilon$ in the stability polynomial, as it was to be established in order to prove \eqref{eq:OscillationConsistencyRestriction}. The term of first degree in $\varepsilon$ in the last display may be computed as
$$
-\varepsilon h \sum_{i=1}^m k_i \begin{bmatrix}
    \sin(h\theta_i)\cos(h(1-\theta_i))&\sin(h\theta_i)\sin(h(1-\theta_i))\\
    \cos(h\theta_i)\cos(h(1-\theta_i))&\cos(h\theta_i)\sin(h(1-\theta_i))
    \end{bmatrix}.
$$
Thus the coefficient of $\varepsilon$ in the stability polynomial is
$$
-\frac{h}{2} \sum_{i=1}^m k_i \Big(\sin(h\theta_i)\cos(h(1-\theta_i))+\cos(h\theta_i)\sin(h(1-\theta_i)\Big)= -\frac{h}{2}
\sum_{i=1}^m k_i \sin(h) = -\frac{h}{2}\sin(h),
$$
as was to be proved.
\end{proof}

\subsection{Stability of the  integrators RKR and KRK}
We now study the stability of RKR/KRK with stability polynomial/semitrace (see \eqref{eq:MatrixRKR}--\eqref{eq:MatrixKRK}):
\begin{equation}\label{eq:stabpolRKR}
P(\varepsilon,h)= \cos(h) -\frac{h\varepsilon}{2} \sin(h).
\end{equation}
The conditions $P(\varepsilon,h) = 1$ and $P(\varepsilon,h) = -1$ correspond to $\varepsilon =\alpha(h)$ and $\varepsilon =\beta(h)$ respectively with
\begin{equation}\label{eq:alphabeta}
\alpha(h) = -\frac{2}{h} \tan\left(\frac{h}{2}\right), \qquad  \beta(h) = \frac{2}{h} \cot\left(\frac{h}{2}\right).
\end{equation}
If we restrict  attention to $0<h<\pi$, then the condition $|P(\varepsilon,h)|\leq 1$ holds if and only if $\varepsilon\in[\alpha(h),\beta(h)]$;  also, for such values of $h$, $\alpha(h)<-1$, $0< \beta(h)$. When integrating the model problem (where $\varepsilon>-1$) we have stability for $\varepsilon \in (-1, \beta(h))$ and exponential instability for $\varepsilon > \beta(h)$. The case $\varepsilon =\beta(h)$ yields linear instability. The function $\beta(h)$ decreases monotonically for $h\in(0,\pi)$ and
therefore increasing $h$ results in a decrease of the interval  $(0,\beta(h))$ of positive values of $\varepsilon$ leading to a stable integration.  As $h\uparrow\pi$, we have $\beta(h)\downarrow 0$, the interval $(0,\beta(h))$ approaches the empty set and thus there is little interest in considering $h\geq \pi$ when dealing with RKR and KRK. {This coincides with the analysis in \cite[\S 4.2.1]{LeimkuhlerBook}, where it is shown that $h=\pi$ is unstable for any non-zero $\varepsilon$.}

Since, as pointed out before, all (consistent) one-stage integrators are conjugate to  RKR or KRK the discussion above also applies to them. In particular, their  stability polynomial is also given by \eqref{eq:stabpolRKR} (a conclusion that may be reached alternatively from \cref{prop:consistency}, after taking into account that for $m=1$ the stability polynomial is of first degree in $\varepsilon$, so that the term $\mathcal{O}(\varepsilon^2)$ in
\eqref{eq:OscillationConsistencyRestriction} must vanish).

\subsection{The  RKRm and KRKm integrators}
\label{sec:RKRm}
To avoid duplications, the presentation in this subsection is limited to RKR, but all the results apply to KRK in an obvious manner.

In the analysis in the next section we shall use the auxiliary $m$-stage  integrator
$$
\psi_h^{[RKRm]} = \Big(\psi_{h/m}^{[RKR]}\Big)^m;
$$
a single step of length $h$ of $\psi^{[RKRm]}$ demands performing $m$ consecutive substeps with $\psi^{[RKR]}$, each of steplength $h/m$. As a consequence, integrations with $\psi^{[RKRm]}$ are in fact nothing but $\psi^{[RKR]}$ integrations; $\psi^{[RKRm]}$  is just a mathematical construction to facilitate the fair comparison between  $m$-stage integrators (with $m$ evaluations of $f$ per step) and the one-stage $\psi^{[RKR]}$ (with only one evaluation of $f$ per step).

Clearly
$$
M^{[RKRm]}_{\varepsilon, h} = \Big(M^{[RKR]}_{\varepsilon, h/m}\Big)^m,
$$
and, for the eigenvalues, $\lambda^{[RKRm]}_{\varepsilon,h}=\Big(\lambda^{[RKR]}_{\varepsilon,h/m}\Big)^m$.
It follows easily from \eqref{eq:alphabeta} that, restricting attention to $h<m\pi$, $|P^{[RKRm]}|<1$ if and only if
$\varepsilon\in(\alpha_m(h),\beta_m(h))$ with
\begin{equation}\label{eq:alphabetam}
\alpha_m(h) = -\frac{2m}{h} \tan\left(\frac{h}{2m}\right)<-1, \qquad  \beta_m(h) = \frac{2m}{h} \cot\left(\frac{h}{2m}\right)>0.
\end{equation}
When integrating the model problem, RKRm is stable if and only if $\varepsilon\in (\alpha_m(h),\beta_m(h))$ (although, as mentioned above, only stability for $\varepsilon>-1>\alpha_m(h)$ is significant). The case
$\varepsilon > \beta_m(h)$ yields exponential instability and $\varepsilon = \beta_m(h)$ gives linear instability. See Figure~\ref{fig:proof}.

 We now find an expression for the stability polynomial $P^{[RKRm]}(\varepsilon, h)$.
 Write $\lambda^{[RKR]}_{\varepsilon,h} = \exp({\rm i} \theta_{\varepsilon,h})$ ($\theta$ is real if $\lambda$ has unit modulus) with $\rm i$ the imaginary unit. Then, recalling \eqref{eq:stabpolRKR}, we may write
$$\cos(h) -\frac{h\varepsilon}{2} \sin(h)=
P^{[RKR]}(\varepsilon, h) = \frac{1}{2} \left(\lambda^{[RKR]}_{\varepsilon,h}+\frac{1}{\lambda^{[RKR]}_{\varepsilon,h}}\right)
= \frac{1}{2} \big(\exp ({\rm i} \theta_{\varepsilon,h})+\exp ({-}{\rm i} \theta_{\varepsilon,h})\big)
= \cos(\theta_{\varepsilon,h}),
$$
and
$$
P^{[RKRm]}(\varepsilon, h) = \frac{1}{2} \left(\lambda^{[RKRm]}_{\varepsilon,h}+\frac{1}{\lambda^{[RKRm]}_{\varepsilon,h}}\right)
= \frac{1}{2} \big(\exp ({\rm i}m \theta_{\varepsilon,h/m})+\exp ({-}{\rm i}m \theta_{\varepsilon,h/m})\big)
= \cos(m \theta_{\varepsilon,h/m}),
$$
so that, introducing the standard Chebyshev polynomial of the first kind $T_m$ with $T_m(\cos(\zeta)) = \cos(m\zeta)$ for all (real or complex) $\zeta$, we conclude that
\begin{equation}\label{eq:stabpolm}
P^{[RKRm]}(\varepsilon, h) = T_m \left(\cos\left(\frac {h}{m}\right) -\frac{h\varepsilon}{2m} \sin\left(\frac {h}{m}\right)\right).
\end{equation}
\section{Main result}\label{sec:main}

In the statement of the main result we denote by $h_m$ the smallest positive root of the equation
$$
\frac{mh}{2}\sin\left(\frac{h}{m}\right) = \cos\left(\frac{\pi}{m}\right)-\cos\left(\frac{h}{m}\right).
$$
For $m=1$, $h_m=\pi$ and, for $m>1$, $h_m <m\pi$. In addition $h_m$ increases
monotonically with $m$ and a
 straightforward Taylor expansion shows that,
 as $m\uparrow \infty$, $h_m = 12^{1/4} \pi^{1/2} m^{1/2}+o(m^{1/2})$. See Table~\ref{tab:my_label}.

\begin{table}[t]
    \centering
    \begin{tabular}{{|c|}*{10}{|c}|}
    \hline
         $m$ & $1$ & $2$ & $3$ & $4$ & $5$ & $6$ & $7$ & $8$ & $9$ & $10$ \\
         \hline
         $h_m$ & $\pi$ & $4.92$ & $5.98$ & $6.85$ & $7.61$ & $8.30$ & $8.93$ & $9.53$ & $10.08$ & $10.61$ \\
         \hline
    \end{tabular}
    \caption{Values of the quantity $h_m$ used in the main theorem.}
    \label{tab:my_label}
\end{table}

\begin{theorem}\label{theo:main}
Define $h_m$ as above. Then:
\begin{itemize}
\item For $h<m\pi$, $\varepsilon > -1$, integrations of the model problem
\eqref{eq:ModelProblem} with either RKRm and KRKm   are exponentially unstable if and only if
$\varepsilon \in (\beta_m(h),\infty)$.
\item Consider an $m$-stage splitting integrator $\psi_h$ of the
form \eqref{eq:OscillationCompositionIntegrator} with stability polynomial different from the stability
 polynomial \eqref{eq:stabpolm} of the integrators RKRm/KRKm. Then, for $h\neq \pi, 2\pi, \dots,(m-1)\pi$
 and $h<h_m$, the (open) set of values of $\varepsilon>-1$ that lead to exponentially unstable integrations
 of the model problem is \emph{strictly larger} than the interval $(\beta_m(h),\infty)$ where RKRm and KRKm show
 exponential instability.
\end{itemize}
\end{theorem}

This result may be restated by saying that for each fixed $h^*$, $h^*<h_m$, $h^*\neq \pi, 2\pi,
\dots,(m-1)\pi$, the intersection of the stability region with the line $h=h^*$ is strictly larger for RKRm
and KRKm than for integrators with stability polynomial different from \eqref{eq:stabpolm}. Before we prove
 Theorem~\ref{theo:main}, we need an auxiliary result that we present in the following subsection.

\subsection{Chebyshev polynomials}

It is well known that many properties of the Chebyshev polynomials are a consequence of the following equioscillation property: $T_m(\xi_i) = (-1)^i$
at the points $\xi_i=\cos(i\pi/m)$, $i=0,\dots,m$,  that partition $[-1,1]$ as $-1= \xi_m<\xi_{m-1}<\dots <
\xi_1< \xi_0 = 1$. We shall need  the following well-known, elementary equioscillation result, whose proof we provide for completeness:

\begin{lemma}Consider $k+1$ real points $x_0<x_1<\dots<x_k$. If $Q$ is a real polynomial such that either
$$Q(x_i) \geq 0,\: i\: {\rm  even}\quad {\rm and} \quad Q(x_i)\leq 0,\: i\: {\rm  odd},$$
or
$$Q(x_i) \leq 0,\: i\: {\rm  even}\quad {\rm and} \quad Q(x_i)\geq 0,\: i\: {\rm  odd},$$
then $Q(x)$ has $\geq k$ zeros (counting multiplicities) in the interval $[x_0,x_k]$.
\end{lemma}

\begin{proof}Consider the $k$ disjoint intervals
$$
J_1= [x_0,x_1),\: J_2= [x_1,x_2),\dots, J_{k-1}=[x_{k-2},x_{k-1}),\;J_k= [x_{k-1},x_k],
$$
that partition $[x_0,x_k]$. We first point out that $Q(x)$ must have at least a zero in the closed interval $J_k$ (otherwise $Q(x)$ would be strictly $>0$ or strictly $<0$ for  $x\in[x_{k-1},x_k]$, in contradiction with the hypothesis).
On the other hand, it is possible that some of the semiclosed $J_i$, $i=1,\dots,k-1$, contain no zero of $Q(x)$, but, if that is the case, then $Q(x_i) = 0$. Furthermore, in that case, $J_{i+1}$ must contain at least two zeros, for if it only contained a single zero at $x_i$, then either $Q(x_{i-1})>0$, $Q(x_{i+1})<0$ or $Q(x_{i-1})<0$, $Q(x_{i+1})>0$. Thus, if a subinterval other than $J_k$ carries no zero, then the one to its right carries two, and this gives a total of at least $k$ zeros.
\end{proof}

The following result on Chebyshev polynomials is to our best knowledge not available in the literature. Its
proof is based on the preceding lemma.

\begin{proposition}\label{prop:chebysvev}
For given $m\geq 2$, let $P(x)$ be a real polynomial of degree $\leq m$ different from $T_m(x)$.
Assume that $P(x)-T_m(x)$ has a double zero $\xi\in (-1,1)$ such that $\xi\neq \xi_i=\cos(i\pi/m)$ for $i = 1,\dots,m-1$.
Then $|P(x)|>1$ for some  $x\in(\xi_m, \xi_1)= (-1,\cos(\pi/m))$.
\end{proposition}
\begin{proof}Assume that $|P(x)|\leq 1$  in $(\xi_m, \xi_1)$ and consider the difference $D(x) =
P(x)-T_m(x)$. For $i= 1,\dots,m$ with $i$ odd, we have  $ D(\xi_i) = P(\xi_i)-T(x_i) = P(\xi_i)- (-1)\geq -1+1
= 0$. Similarly, for $i= 1,\dots,m$ with $i$ even, we have $D(\xi_i) \leq 0$. There are two cases:
\begin{enumerate}
\item $\xi \in (\xi_1,\xi_0)$. Then, by the lemma, $D(x)$ has $\geq m-1$ zeros in $[\xi_m,\xi_1]$. These and
    the double zero $\xi\in(\xi_1,\xi_0)$ provide $\geq m+1$ zeros of $D(x)$. It follows  that $D(x)$
    vanishes identically, in contradiction with the hypotheses of the proposition.
\item $\xi$ is in an interval $(\xi_{j+1},\xi_j)$ with $j=1,\dots, m-1$. By applying the lemma twice, we see
    that $D(x)$ has  $\geq j-1$ zeros in $[\xi_j,\xi_1]$ and  $\geq m-j-1$ zeros in $[\xi_m,\xi_{j+1}]$. The
    subinterval $(\xi_{j+1},\xi_j)$ must contain at least three zeros, because, if the multiplicity of $\xi$
    were exactly $2$ and there were no other zeros in the subinterval, then $D(\xi_j)$ and $D(\xi_{j+1})$ would
    be either both $>0$ or both $<0$. We have thus found  $\geq j-1+(m-j-1)+3 = m+1$ zeros, which again leads to a
    contradiction.
\end{enumerate}
\end{proof}
\subsection{Proof of the main result}
\begin{figure}[t]
\centering
\includegraphics[width=.8\textwidth]{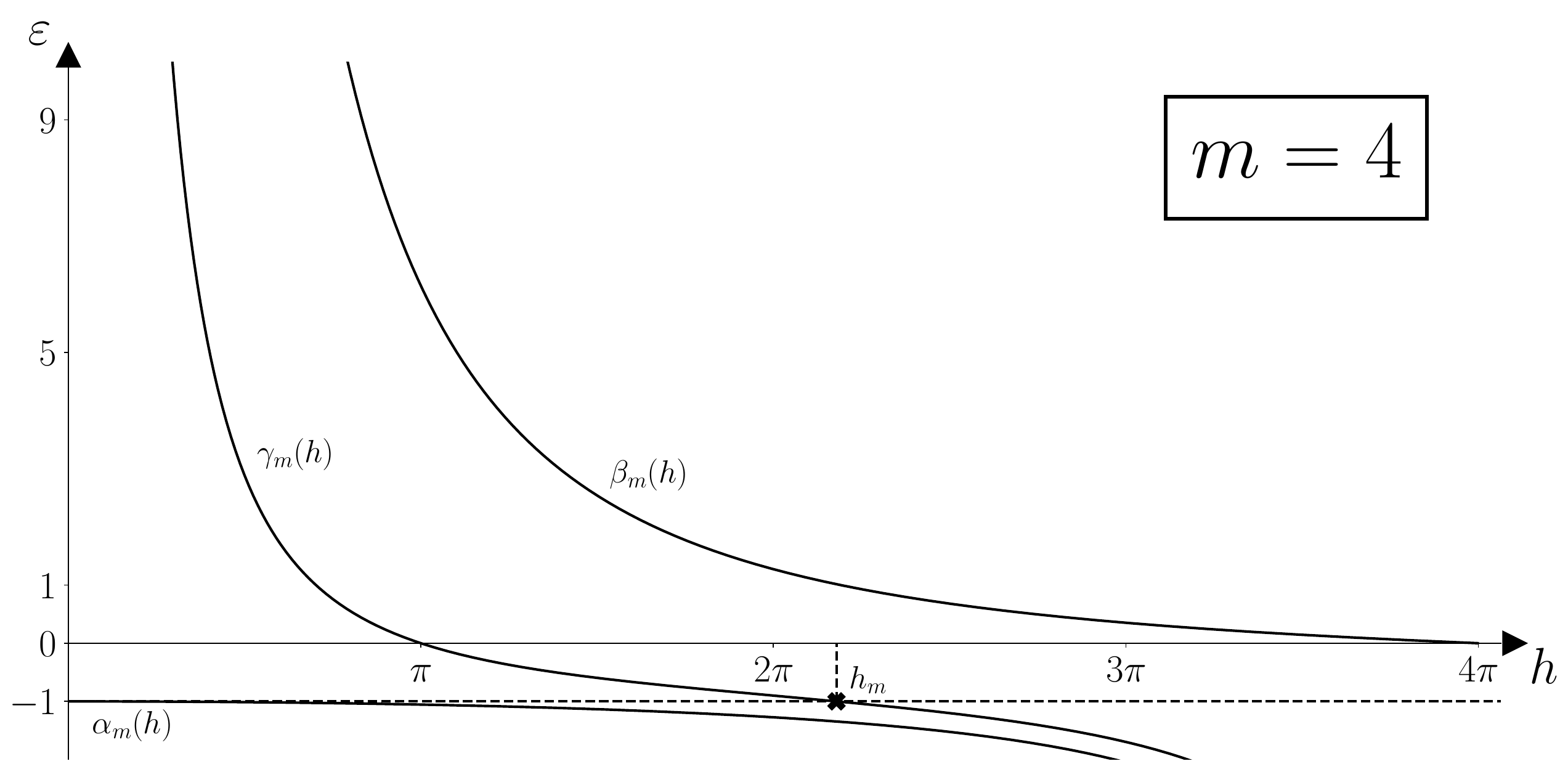}
\caption{Proof of the main result in the case $m=4$. In the model problem $\varepsilon >-1$. RKR4 and KRK4 are stable in the open region bounded by the lines $h= 0$, $h=m\pi$, $\varepsilon =- 1$, $\varepsilon= \beta_m(h)$. For each fixed $h$ such that $h<4\pi$, $h\neq \pi, 2\pi, 3\pi$, a competitor integrator will have $|P(\varepsilon,h)| >1$  for some $\varepsilon \in (\gamma_h,\beta_h)$. When $h<h_m$, those values
of $\varepsilon$ are $>-1$.
}
\label{fig:proof}
\end{figure}
The first item in Theorem~\ref{theo:main} was established at the very end of Section~\ref{sec:preliminaries}. In the second item, we only have to deal with $m\geq 2$, because we also saw in Section~\ref{sec:preliminaries} that there is no consistent one-stage integrator with stability polynomial different from the stability polynomial \eqref{eq:stabpolRKR} of RKR or KRK.

With fixed $h$ satisfying the conditions of the theorem, we change variables replacing $\varepsilon$ by the new variable
$$
x= \cos\left( \frac{h}{m}\right)-\frac{\varepsilon h}{2m} \sin\left( \frac{h}{m}\right).
$$
Since $h<h_m<m\pi$, this transformation is bijective. It maps $\varepsilon = \alpha_m(h)$ (see \eqref{eq:alphabetam}) into $x=1$ and $\varepsilon = \beta_m(h)$  into $x=-1$. The change of variables is chosen in such a way that, according to \eqref{eq:stabpolm}, the stability polynomial of RKRm or KRKm is transformed into the Chebyshev polynomial $T_m(x)$.

Denote by $P(x)$ the $m$-degree polynomial in the variable $x$ resulting from changing variables in
 the stability polynomial $P(\varepsilon,h)$ of the integrator $\psi_h$
 (note that the dependence of $P(x)$ on $h$ has been left out of the notation). By hypothesis, $P(x)$ cannot coincide with $T_m(x)$.
 From Proposition~\ref{prop:consistency}, $P(\varepsilon,h)-P^{[RKRm]}(\varepsilon,h)$ as a polynomial in $\varepsilon$
  has a double root at $\varepsilon=0$ and accordingly $P(x)-T_m(x)$ has a double zero at the
  corresponding value of $x$ given by $\xi=\cos(h/m)$. Since $h$ is assumed to be $\neq \pi, 2\pi,\dots,(m-1)\pi$, $\xi$ is not one of the extrema
$\xi_i = \cos(i\pi/m)$, $i = 1,\dots,m-1$, of $T_m(x)$. Proposition~\ref{prop:chebysvev} reveals that $|P(x)|$
has to exceed $1$ as some point $x\in (-1, \cos(\pi/m))$; the corresponding $\varepsilon$-value will be in the
interval $(\gamma_m(h), \beta_m(h))$ with
$$
\gamma_m(h) = \frac{2m}{h\sin(h/m)}\Big(\cos(h/m)-\cos(\pi/m)\Big).
$$
The condition $h<h_m$ implies $\gamma_m(h) >-1$ (see Figure~\ref{fig:proof}). We have thus found values of $\varepsilon\in (-1,\beta_m(h))$ that lead to instability and the proof is complete.

\section{Assessing the size of the stability region}
\label{sec:discussion}

The result we have just presented does not provide quantitative information on the size of stability regions in the full $(\varepsilon,h)$ plane
of the different integrators.  In this section, we present a more quantitative analysis; it turns out that
Strang integrators have much larger stability regions than their competitors.
\subsection{Stability near $\varepsilon = 0$, $h=n\pi$}
When $\varepsilon = 0$ all splitting integrators \eqref{eq:OscillationCompositionIntegrator} are exact and therefore $M_{0,h}$ is the matrix corresponding to a rotation by $h$ radians, with semitrace $P(0,h) =\cos(h)$. If $h>0$ is not an integer multiple of $\pi$,  the magnitude of the trace is $<2$ and the matrix
$M_{0,h}$ is strongly stable \cite[sections 25 and 42]{ArnoldBook} and \cite{Krein1950} (see also \cite{BouRabee2017}). Accordingly, the integrator is stable in a neighborhood of $(0,h)$. On the other hand, $P(0,n\pi) = (-1)^n$, $n=1,2,\dots$, and perturbations of the parameter values $\varepsilon = 0$, $h =n\pi$ may render the integrator exponentially unstable. For instance, RKRm and KRKm are stable, as we know, in the neighbourhood of
$(0,\pi)$, \dots, $(0, (m-1)\pi)$ but not in the neighbourhood of $(0,m\pi)$ (see Figure~\ref{fig:proof}). We now investigate the stability of  general integrators \eqref{eq:OscillationCompositionIntegrator} in the neighbourhood of the points $(0,n\pi)$, $n = 1, 2, \dots$

We assume that $n$ is \emph{odd} (the case $n$ even is entirely parallel). Then $P(0,n\pi) = -1$ and a necessary condition for the method to be stable in a neighbourhood of $(0,n\pi)$ is that this point be a minimum of $P$. Since, for $\varepsilon = 0$, $P(0,h) = \cos(h)$, we have $(\partial/\partial h) P(0,h) = -\sin(h)$ and  $(\partial/\partial h) P(0,n\pi) = 0$. In addition, from Proposition~\ref{prop:consistency}, $(\partial/\partial \varepsilon) P(0,h) = -(h/2) \sin(h)$, and, therefore
$(\partial/\partial \varepsilon) P(0,n\pi)) = 0$; we conclude that all integrators satisfy the first-order necessary conditions for $(0,n\pi)$ to be a minimum of $P$. Turning now to the second-order necessary conditions, from $(\partial^2/\partial h^2)P(0,h) = -\cos(h)$ and $(\partial^2/\partial \varepsilon \partial h ) P(0,h) = (-1/2) (\sin(h)+h\cos(h))$, we see that the Hessian of $P$ at $(0,nh)$ takes the form
$$
    \begin{bmatrix}
    \frac{\partial^2}{\partial \varepsilon^2}P(0,n\pi)&\frac{n\pi}{2}\\
    \frac{n\pi}{2}&1
    \end{bmatrix}.
$$
(The top left entry changes with the integrator, the other three do not.)
For $(0,n\pi)$ to be a minimum, the Hessian has to be positive semidefinite; since the bottom right entry is $>0$, positive semidefiniteness is equivalent to nonnegative determinant, i.e.\ to
$$
\frac{\partial^2}{\partial \varepsilon^2}P(0,n\pi) \geq \frac{n^2\pi^2}{4}.
$$
However  Proposition~\ref{prop:bound} ensures that the opposite inequality holds and we have proved the $n$ odd case of the following result (the $n$ even case is proved in a parallel way, changing minimum to maximum, etc.).

\begin{proposition}
\label{prop:necessary} Assume that an integrator of the form \eqref{eq:OscillationCompositionIntegrator}
 is stable for values of $(\varepsilon,h)$ in a neighbourhood of $(0,n\pi)$, $n= 1,2, \dots$ Then necessarily:
 $$
 \frac{\partial^2}{\partial \varepsilon^2}P(0,n\pi) = (-1)^{n+1}\frac{n^2\pi^2}{4}.
 $$
\end{proposition}

This proposition is helpful to identify suitable values of the parameters $r_i$ and $k_j$ in \eqref{eq:OscillationCompositionIntegrator}, as will be clear in our
study of the stability of the families of three-stage integrators.

\subsection{Palindromic methods with $m=3$ stages}
 Integrators with three or fewer stages are important because, arguably, integrators with four or more stages are too complicated to be used in most applications. For the case of the {kinetic/potential} split systems \eqref{eq:DriftSystem}--\eqref{eq:KickSystem}, there are $3$-stage integrators that clearly improve on Verlet in HMC and molecular dynamics \cite{Fernandez2016, Nishimura2020, Mannseth2018, Aleardi2020, Goth2022, Ahmed2015}. As we shall prove presently, for the \eqref{eq:RotSystem}--\eqref{eq:KickSystemf} splitting studied in this paper, there is little room for improving on the Strang splitting. {As explained in the introduction this result is very relevant when choosing the integrator for HMC algorithms to sample from target distributions resulting from perturbing a Gaussian.}
\begin{figure}[t]
    \includegraphics[width=\textwidth]{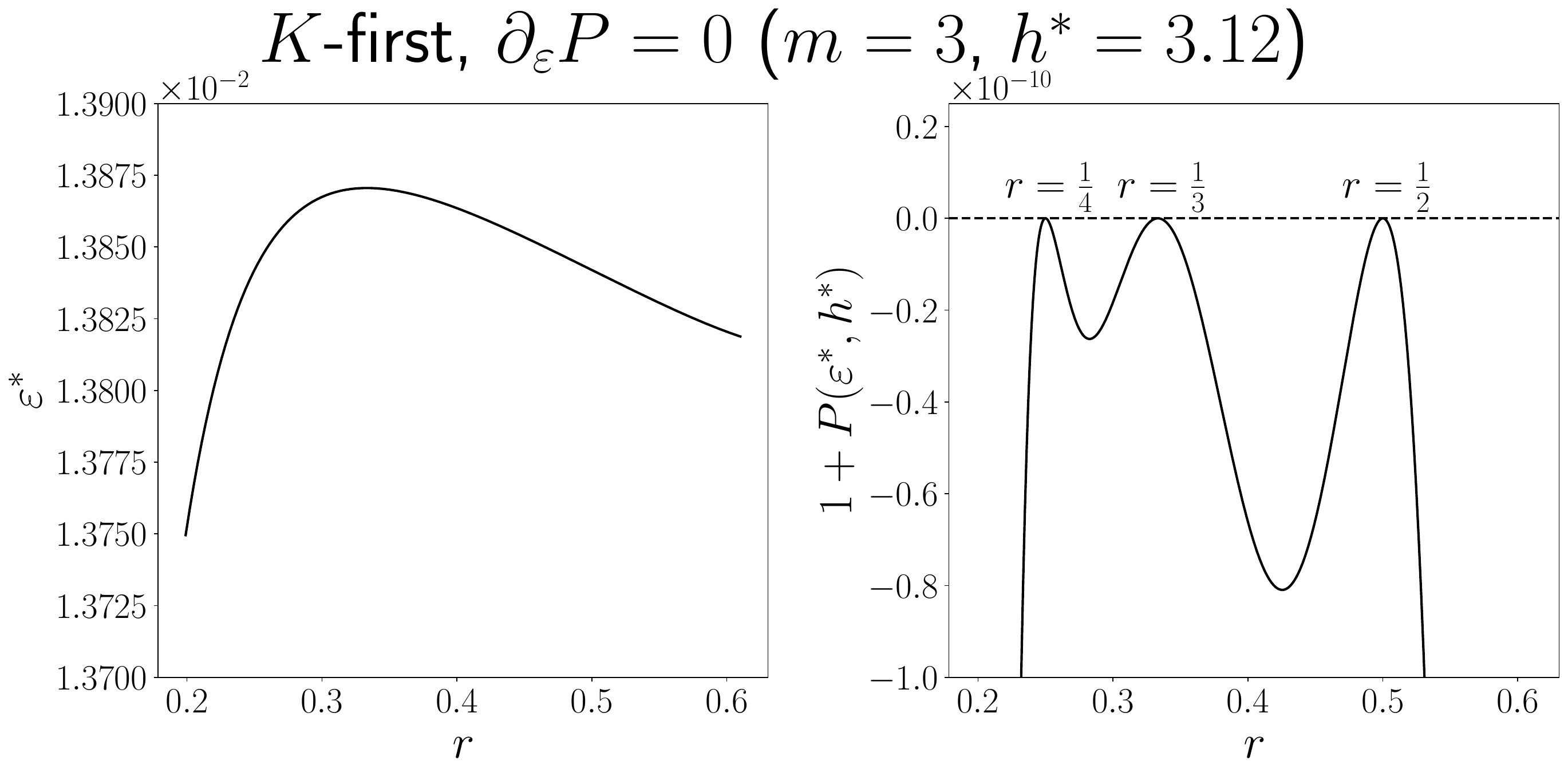}
    \caption{Palindromic three-stage, K-first integrators. On the left, for each $r$, the value $\varepsilon^*$ of the parameter $\varepsilon$} that locally minimizes the stability polynomial $P(\varepsilon, h^*,r)$. On the right, the minimum value $P(\varepsilon^*, h^*,r)$ as a function of $r$: except for three exceptional cases (see text), all integrators show $P<-1$, i.e.\ exponential instability.
    \label{fig:dPde}
    \end{figure}

For the sake of brevity we only present our findings for the K-first case in \eqref{eq:OscillationCompositionIntegrator}. The results for the R-first case differ in the details but yield the same conclusions. As we have noted several times, it is sufficient to study the palindromic case, for which, after imposing consistency, integrators take the form
\begin{equation}\label{eq:threestage}
\psi_h=\varphi_{kh}^{[K]}\circ\varphi_{rh}^{[R]}\circ\varphi_{(1/2-k)h}^{[K]}
\circ\varphi_{(1-2r)h}^{[R]}\circ\varphi_{(1/2-k)h}^{[K]}\circ\varphi_{rh}^{[R]}\circ\varphi_{kh}^{[K]}.
\end{equation}
There are two free parameters $k$ and $r$. If we wish to have stability in a neighbourhood of $(0,\pi)$ in the $(\varepsilon,h)$ plane, we have to impose the necessary condition in Proposition~\ref{prop:necessary},  that for \eqref{eq:threestage} is found to read
$$
4 k\sin^2(\pi r)= -\cos(2\pi r).
$$
However, this condition is only necessary for $P$ to have a minimum $P=-1$ at $\varepsilon = 0$, $h=\pi$. To investigate the behaviour of $P$ in the neighbourhood of $(0,\pi)$, we proceed as follows. We use the last display to express $k$ in terms of $r$ and see $P$ as a function of $(\varepsilon, h, r)$. We then fix a value $h^*= 3.12$ of $h$ slightly below $\pi$ and look at the behaviour of $P(\varepsilon, h^*, r)$. For each $r$ in a suitable range,\footnote{We present results for $r\in [0.2,0.6]$. Values of $r$ outside this interval are not of interest as a preliminary computer search shows they have poor stability properties near $\varepsilon = -1$.} we identify the value $\varepsilon^*(h^*,r)\approx 0$ of $\varepsilon$ for which
$(\partial/\partial \varepsilon)P(\varepsilon,h^*,r)$ vanishes (and therefore the function $\varepsilon\mapsto P(\varepsilon, h^*,r)$ may achieve a minimum) and plug this value into $P$ to obtain a function $F(r)=P(\varepsilon(h^*,r),h^*,r)$ of the real variable $r$. This function is plotted in the right panel of Figure~\ref{fig:dPde}, where we see that for \lq\lq most\rq\rq\ values of $r$, $F(r)$ takes values below $-1$, indicating exponential instability of the integrator. There are however three exceptional values of $r$, where $F=-1$:
\begin{itemize}
\item $r=1/4$. This leads to $k=0$ so that the first and last kicks in \eqref{eq:threestage} are the identity and may be suppressed. The integrator is then seen to be RKR2, that we know is indeed stable in the neighbourhood of $(0,\pi)$.
\item $r=1/3$. This yields KRK3, that we know is stable in the neighbourhood of $(0,\pi)$ (and also in the neighbourhood of $(0,2\pi)$).
\item $r=1/2$. Now the central rotation in \eqref{eq:threestage} is the identity. The integrator is KRK2, that we know is stable in the neighbourhood of $(0,\pi)$.
\end{itemize}

The values of $\varepsilon^*$ where the algorithm has been found to be exponentially unstable are plotted in the left panel of Figure~\ref{fig:dPde}. This shows that, for $h=3.12$, all the integrators considered (with the exceptions of RKR2, KRK2 and KRK3) are unstable for values of $\varepsilon$ extremely close to 0. For comparison, using \eqref{eq:alphabetam}, one sees that for $h=3.12$, RKR3 and KRK3  are stable for $\varepsilon \in (-1,3.36)$ and
RKR2, KRK2 are stable for $\varepsilon \in (-1,1.30)$. Also, from Theorem~\ref{theo:main}, for fixed, very small $\varepsilon>0$, RKR3 and KRK3  are stable up to $h \approx 3\pi$, while most three stage integrators have lost stability before $h$ reaches $\pi$.
The conclusion is clear: three-stage splitting integrators different from Strang have  very limited stability domains.
\section{A technical result}
In this section we establish the following result that was used to prove Proposition~\ref{prop:necessary}:

\begin{proposition}
\label{prop:bound}The stability polynomial $P(\varepsilon,h)$ of any (consistent) splitting integrator \eqref{eq:OscillationCompositionIntegrator} satisfies
$$
\frac{\partial^2}{\partial \varepsilon^2}P(0,n\pi)\leq \frac{n^2\pi^2}{4},\qquad   n = 1,3, \dots,
$$
and
$$
\frac{\partial^2}{\partial \varepsilon^2}P(0,n\pi)\geq -\frac{n^2\pi^2}{4},\qquad   n = 2,4, \dots
$$
\end{proposition}
\begin{proof}We recommence from \eqref{eq:matrixasproduct} in the proof of Proposition~\ref{prop:consistency}. The coefficient of $\varepsilon^2$ in the right hand-side of that equality is, with $\eta_i = \sum_{n=i+1}^{m+1}r_n$, $\theta_j = \sum_{n=1}^j r_n$,
$$
h^2 \sum_{i=2}^m \sum_{j=1}^{i-1}k_ik_j \exp(\eta_i hR) K \exp((1-\eta_i-\theta_j) hR) K \exp(\theta_j hR),
$$
where, by using the expressions for $\exp(tR)$ and $K$, the product of matrices in the summation may be computed as
\begin{gather*}
\begin{bmatrix}
   -\big(\sin(h\theta_j)\cos(h(1-\eta_i-\theta_j))-\sin(h(1-\eta_i))\big)\sin(h\eta_i)&\cdots\\
    \cdots&-\big(\sin(h\eta_i)\cos(h(1-\eta_i-\theta_j))-\sin(h(1-\theta_j))\big)\sin(h\theta_j)
    \end{bmatrix}.
\end{gather*}
We next take semitraces and recall that, from Taylor's theorem, the coefficient of $\varepsilon^2$
in a polynomial equals twice its second derivative evaluated at $\varepsilon=0$. In this way we find
\begin{eqnarray*}
\frac{\partial^2}{\partial \varepsilon^2}P(0,h)&=&-h^2 \sum_{i=2}^m \sum_{j=1}^{i-1}k_ik_j\Big[\big(\sin(h\theta_j)\cos(h(1-\eta_i-\theta_j))-\sin(h(1-\eta_i))\big)\sin(h\eta_i)\\
&&\qquad\qquad\qquad\qquad\qquad\qquad+\big(\sin(h\eta_i)\cos(h(1-\eta_i-\theta_j))-\sin(h(1-\theta_j))\big)
\sin(h\theta_j)\Big].
\end{eqnarray*}
By transforming the products of trigonometric functions into sums, we obtain
$$
\frac{\partial^2}{\partial \varepsilon^2}P(0,h) = \frac{h^2}{2}
\sum_{i=2}^m \sum_{j=1}^{i-1}k_ik_j \left(\cos\left(2h \left(\frac{1}{2}-\eta_i-\theta_j\right)\right) -\cos(h)\right),
$$
and evaluating at $h=n\pi$ we find, after some additional trigonometric manipulations,
$$
\frac{\partial^2}{\partial \varepsilon^2}P(0,n\pi) = (-1)^{n+1} n^2\pi^2
\sum_{i=2}^m \sum_{j=1}^{i-1}k_ik_j \sin^2\big(n\pi(\eta_i+\theta_j)\big).
$$
We now note that $\eta_i+\theta_j = 1-(\theta_i-\theta_j)$ and $\sin^2(n\pi-(\theta_i-\theta_j))
=\sin^2(n\pi(\theta_i-\theta_j))$, so that
$$
\frac{\partial^2}{\partial \varepsilon^2}P(0,n\pi) = (-1)^{n+1} n^2\pi^2
\sum_{i=2}^m \sum_{j=1}^{i-1}k_ik_j \sin^2\big(n\pi(\theta_i-\theta_j)\big).
$$
The proof will be ready if we prove that
$$\sum_{i=2}^m \sum_{j=1}^{i-1}k_ik_j \sin^2\big(n\pi(\theta_i-\theta_j)\big)\leq \frac{1}{4},$$
or, writing the double sum in a more symmetric form,
$$S=\sum_{i=1}^m \sum_{j=1}^{m}k_ik_j \sin^2\big(n\pi(\theta_i-\theta_j)\big)\leq \frac{1}{2}.$$
At this point, it is convenient to assume that (i) $m$ is even and (ii) the integrator is palindromic. As noted before there is no loss of generality in assuming (ii). And $m$ may always be taken to be even
by adding dummy stages. The double sum $S$ may be decomposed as
$$
S = \sum_{i=1}^m \sum_{j=1}^{m}
= \sum_{i=1}^{m/2} \sum_{j=1}^{m/2}
+ \sum_{i=1}^{m/2} \sum_{j=m/2}^{m}
+ \sum_{i=m/2}^m \sum_{j=1}^{m/2}
+ \sum_{i=m/2}^m \sum_{j=m/2}^{m},
$$
which, by symmetry, implies
$$
S = 2 \sum_{i=1}^{m/2} \sum_{j=1}^{m/2}k_ik_j \sin^2\big(n\pi(\theta_i-\theta_j)\big)
+ 2\sum_{i=1}^{m/2} \sum_{j=m/2}^{m}k_ik_j \sin^2\big(n\pi(\theta_i-\theta_j)\big).
$$
and, since $k_{m+1-j} = k_j$, $\theta_i-\theta_{m+1-j}=\theta_i+\theta_j-1$, $\sin^2\big(n\pi(\theta_i+\theta_j-1)\big) = \sin^2\big(n\pi(\theta_i+\theta_j)\big)$,
$$
S = 2 \sum_{i=1}^{m/2} \sum_{j=1}^{m/2}k_ik_j \sin^2\big(n\pi(\theta_i-\theta_j)\big)
+ 2\sum_{i=1}^{m/2} \sum_{j=1}^{m/2}k_ik_j \sin^2\big(n\pi(\theta_i+\theta_j)\big).
$$
We finally invoke the trigonometric identity $\sin^2(A+B)+\sin^2(A-B) = 1-\cos(2A)\cos(2B)$ and write
\begin{eqnarray*}
S &=& 2 \sum_{i=1}^{m/2} \sum_{j=1}^{m/2} k_ik_j \Big(1-\cos(2n\theta_i)\cos(2n\theta_j)\Big)\\
&=& 2 \left(\sum_{i=1}^{m/2}k_i\right)^2 - 2 \left(\sum_{i=1}^{m/2}k_i \cos (2n\pi\theta_i)\right)^2\\
&\leq&  2 \left(\sum_{i=1}^{m/2}k_i\right)^2 = \frac{1}{2},
\end{eqnarray*}
and the proof is complete.
\end{proof}

\section*{Acknowledgments}
\noindent
We would like to thank Sergio Blanes for alerting us to the open problem of the Strang splitting's optimality for the alternative $R,K$ integrators.

\bigskip\bigskip\bigskip

\bibliographystyle{siamplain}

\end{document}